\documentclass[11pt,reqno]{amsart}

\usepackage[T1]{fontenc}
\usepackage{lmodern}
\usepackage{microtype}
\usepackage{amsmath,amssymb,mathtools,mathrsfs}
\usepackage{hyperref}
\usepackage{graphicx}
\usepackage{orcidlink}
\hypersetup{
  colorlinks=true,
  linkcolor=blue,
  citecolor=blue,
  urlcolor=blue
}

\newtheorem{theorem}{Theorem}[section]
\newtheorem{proposition}[theorem]{Proposition}
\newtheorem{lemma}[theorem]{Lemma}

\theoremstyle{definition}

\theoremstyle{remark}

\newcommand{\C}{\mathbb C}
\newcommand{\R}{\mathbb R}
\newcommand{\M}{\mathcal M}
\newcommand{\MK}{\mathfrak K}
\newcommand{\KM}{\MK_{\M}}

\newcommand{\dist}{\operatorname{dist}}
\newcommand{\pr}{\pi}
\newcommand{\bx}{\mathbf{x}}

\newcommand{\LimSup}{\operatorname*{\overline{lim}}}
\newcommand{\LimInf}{\operatorname*{\underline{lim}}}

\newcommand{\g}{\mathbf{g}}
\title[Filled Julia Set Bundle over the Mandelbrot Set]{Local Connectivity of the Filled Julia Set Bundle over the Mandelbrot Set}
\author[Yueyang Wang]{Yueyang Wang\,\orcidlink{0009-0007-5661-8896}}
\address{School of Mathematics, Shanghai University of Finance and Economics, Shanghai, 200433, China}
\email{yueyangwang@icloud.com}
\date{}
\subjclass[2020]{Primary 37F10; Secondary 37F20, 37F35, 37F46, 54D05}
\keywords{Quadratic polynomial, filled Julia set, Mandelbrot set, parabolic implosion, holomorphic motion, local connectedness, Yoccoz puzzle, parapuzzle}

\begin{document}

\begin{abstract}
For the quadratic family $f_c(z)=z^2+c$, we study the non-escaping locus
$\KM=\{(c,z):c\in\M,\ z\in K_c\}$, the model space for cubic
capture straightening in the framework of Inou and Kiwi. Here $\M$ is
the Mandelbrot set, $K_c$ is the filled Julia set of $f_c$, and
$J_c=\partial K_c$.
We prove that $\partial\KM$ has full Hausdorff dimension $4$.
The set $\KM$ is locally connected at every $(c,z)$ with hyperbolic
$c\in\M$. At parameters on $\partial\M$, we construct mixed puzzle
pieces from Yoccoz puzzles and parapuzzles. For non-degenerate pairs,
shrinking of both planar puzzles implies local connectedness of
$\KM$ at $(c,z)$. However, local connectedness fails at some
$(c,z)$ with $c$ real parabolic and $z\in K_c\setminus J_c$,
even though both the parameter space and the filled Julia fiber
are locally connected there.
\end{abstract}

\maketitle

\section{Introduction}
\subsection{Background}
The quadratic family $f_c(z)=z^2+c$, $c\in\C$, has been a central
object of study in complex dynamics for several decades; see, for example,
\cite{DH,DHPolynomialLike,HubbardYoccoz,McMullen,Lyubich12,Shishikura,MilnorLC,AKLS,KahnLyubichLC,BuffCheritat}.
Let $K_c$ and $J_c=\partial K_c$ denote its filled Julia set and
Julia set, and let $\M$ be the Mandelbrot set.
We study the topology of the family of filled Julia sets over $\M$,
\[
 \KM=\{(c,z)\in\C^2:c\in\M,\ z\in K_c\}.
\]
The parameter projection has base $\M$ and fiber $K_c$.

An important motivation comes from capture straightening
\cite[Sections~1.13--1.14]{InouKiwi}. The cubic capture schema is
represented by
\[
 \g:\{v_0,v_1\}\times\C\to\{v_0,v_1\}\times\C,
 \quad \g(v_j,w)=(v_0,w^2+c_j),\quad j=0,1.
\]
Its fiberwise connectedness locus consists of parameters for which
both critical orbits, starting at $(v_0,0)$ and $(v_1,0)$, are
bounded. Equivalently, $c_0\in\M$ and $c_1\in K_{c_0}$, so this
locus is $\KM$.

Let $f_0$ be a primitive hyperbolic postcritically finite cubic of
capture type \cite[Sections~1.12--1.13]{InouKiwi}. Here primitive
means that distinct bounded Fatou components have disjoint closures.
Its combinatorial renormalization locus \cite[Section~1.1]{InouKiwi} is
\[
 \mathcal C(f_0)=\{f\in\mathcal C_3:
 \lambda_{\mathbb Q}(f)\supset\lambda_{\mathbb Q}(f_0)\},
\]
where $\mathcal C_3$ is the connectedness locus of monic centered
cubics and $\lambda_{\mathbb Q}$ denotes rational lamination
\cite[Section~3.1]{InouKiwi}.
After fixing an internal angle system \cite[Section~1.9]{InouKiwi},
Inou and Kiwi obtain a
bijection $\chi:\mathcal C(f_0)\to\KM$
\cite[Theorems~D and~F]{InouKiwi}. It is discontinuous
\cite[Corollary~1.1]{InouDiscontinuity}, but maps the hyperbolic
component $\mathcal H(f_0)$ biholomorphically onto
\[
 \mathfrak{H}_0=\{(c,z):c\in\mathcal H_0,\ z\in\operatorname{int}K_c\},
\]
where $\mathcal H_0$ is the main cardioid
\cite[Theorem~C]{InouKiwi}.
Milnor studied this universal capture component and its closure,
in multiplier coordinates, as a model for questions about hyperbolic
component boundaries \cite[pp.~19--20]{MilnorBoundaries}.
He proved that $\overline{\mathfrak{H}_0}$ is not locally connected.

\subsection{Main results}
We place $\KM$ in the larger family
\begin{equation}\label{eq:total}
 \MK=\{(c,z)\in\C^2:z\in K_c\},
\end{equation}
which is the non-escaping locus of the skew product
$F:\C^2\to\C^2$, $F(c,z)=(c,f_c(z))$.
Write $\pr:\C^2\to\C$ for the projection $\pr(c,z)=c$ and
$\MK_A=\MK\cap(A\times\C)$ for $A\subset\C$, so that
$\KM=\MK_{\M}$. Related questions for fibered polynomials and
polynomial skew products are studied in \cite{Sester,Jonsson}.
Here the base dynamics is the identity. Our focus is how the topology
of the total space depends on the topology of the base and fibers,
and on their variation across stable and bifurcating parameter regions.

\begin{theorem}\label{thm:main}
The set $\MK$ is closed, connected, and completely invariant under
$F$, and $\pr|_{\MK}:\MK\to\C$ is proper. The set $\KM$ is
compact and connected.
\end{theorem}

Over hyperbolic parameter regions, $J$-stability and holomorphic motions
\cite{MSS,Slodkowski} give a fiber-preserving local product structure.
The topology of the fiber then determines local connectedness of the
total space.

\begin{theorem}\label{thm:stable-main}
Let $(c,z)\in\MK$. If $c\notin\M$, then $\MK$ is not locally
connected at $(c,z)$. If $c\in\M$ and $f_c$ is hyperbolic, then
$\MK$ is locally connected at $(c,z)$.
\end{theorem}

In particular, $\KM$ is locally connected over every hyperbolic
component of $\M$. To study parameters on $\partial\M$, we use
puzzle and parapuzzle shrinking.

Call a pair $(c_0,z_0)\in\KM$ \emph{non-degenerate} if
$c_0\in\M\setminus\overline{\mathcal H_0}$ and
\[
 f_{c_0}^k(0)\ne\alpha_{c_0},\qquad f_{c_0}^k(z_0)\ne\alpha_{c_0}
 \quad\text{for every }k\ge0,
\]
where $\alpha_{c_0}$ is the repelling dividing fixed point.
Choose a rational wake containing $c_0$ and use the puzzle and
parapuzzle pieces of Section~\ref{sec:puzzles}.

\begin{itemize}
\item \textbf{Dynamical shrinking.}
The nested pieces $P_n^{c_0}(z_0)$ define the impression
\[
 I_{c_0}(z_0)=\bigcap_{n\ge0}\overline{P_n^{c_0}(z_0)}.
\]
The intersections of puzzle pieces with $K_{c_0}$
are connected; see \cite[Section~1, Problem~1-1]{MilnorLC}.
Thus, if $I_{c_0}(z_0)=\{z_0\}$, then $K_{c_0}$ is locally connected at $z_0$
\cite{HubbardYoccoz}.
\item \textbf{Parameter shrinking.}
The nested parapuzzle pieces $\mathcal P_n(c_0)$ define the impression
\[
 \mathcal I(c_0)=\bigcap_{n\ge0}\overline{\mathcal P_n(c_0)}.
\]
If $\mathcal I(c_0)=\{c_0\}$, then $\M$ is locally connected at $c_0$
by \cite[Proposition~3.5]{SchleicherFibers}: the separation fiber
of $c_0$ is contained in $\mathcal I(c_0)$, since the parapuzzle
boundaries are formed by rational ray pairs and equipotentials
\cite{Roesch}.
\end{itemize}
For example, both shrinking properties hold for non-degenerate
pairs with $z_0\in J_{c_0}$ when $f_{c_0}$ is non-renormalizable and its
critical point is recurrent \cite{MilnorLC,Roesch}.

We combine the parapuzzle at $c_0$ and the dynamical puzzle at $z_0$
to construct a \emph{mixed puzzle} at $(c_0,z_0)$. The key ingredient is
the holomorphic motion of dynamical puzzle boundaries over a
parapuzzle piece \cite[Lemma~2.5]{Roesch}.
For a non-degenerate pair $(c_0,z_0)$, let
\[
 h_n:\mathcal P_n(c_0)\times\widehat\C\to\widehat\C
\]
be a sphere extension of the motion of $I_{n+1}^{c_0}$, normalized
at $c_0$ and fixing $\infty$, as provided by S\l{}odkowski's
extension theorem \cite{Slodkowski}; see also
\cite[Theorem~3.1]{Roesch}. Write $h_n^{c}=h_n(c,\cdot)$.
Define the mixed puzzle piece of depth $n$ by
\begin{equation}\label{eq:tube-intro}
 \mathfrak P_n(c_0,z_0)
 =\{(c,w):c\in\mathcal P_n(c_0),\
 w\in h_n^{c}(P_n^{c_0}(z_0))\},
\end{equation}
and its impression by
\[
 \mathfrak I(c_0,z_0)=\bigcap_{n\ge0}\overline{\mathfrak P_n(c_0,z_0)}.
\]
The following theorem relates mixed shrinking to the two planar
shrinking properties.

\begin{theorem}\label{thm:puzzle}
Let $(c_0,z_0)$ be a non-degenerate pair. Then:
\begin{enumerate}
\item[\textup{(i)}] Mixed shrinking is equivalent to simultaneous
parameter and dynamical shrinking:
\[
 \mathfrak I(c_0,z_0)=\{(c_0,z_0)\}
 \quad\Longleftrightarrow\quad
 \left\{\begin{aligned}
  \mathcal I(c_0)&=\{c_0\},\\
  I_{c_0}(z_0)&=\{z_0\}.
 \end{aligned}\right.
\]
\item[\textup{(ii)}] Under these equivalent conditions, $\KM$ is
locally connected at $(c_0,z_0)$.
\end{enumerate}
\end{theorem}

In particular, $\KM$ is locally connected at every non-degenerate
pair $(c,z)$ with $z\in J_c$ for which $f_c$ is non-renormalizable
and has a recurrent critical point.
This raises a natural question: does local connectedness of $\M$
at $c$, together with local connectedness of $K_c$ at $z$, imply
local connectedness of $\KM$ at $(c,z)$?
Parabolic implosion \cite{Lavaurs,DouadyContinuity,DSZ,Shishikura},
specifically the Lavaurs limits in \cite{RojasYampolsky}, yields a
negative answer. We use \emph{primitive parabolic parameter} for
the root of a primitive hyperbolic component, in the sense of
\cite[Section~6]{MilnorPortrait}.

\begin{theorem}\label{thm:parabolic-main}
There exist a primitive parabolic parameter $c_*\in\partial\M$ and
$z_*\in K_{c_*}\setminus J_{c_*}$ such that $\KM$ is not locally
connected at $(c_*,z_*)$.
\end{theorem}

At a parabolic parameter, both $\M$ and $K_c$ are locally connected
\cite{TanParabolic,DH,MilnorBook}; thus the obstruction in
Theorem~\ref{thm:parabolic-main} comes from the variation of the
fibers. The point $z_*$ lies in $K_{c_*}\setminus J_{c_*}$.

Another application of parabolic implosion is Shishikura's theorem
that $\partial\M$ has Hausdorff dimension $2$ \cite{Shishikura}.
His dimension estimates also yield full Hausdorff dimension for
the boundary of $\KM$.

\begin{theorem}\label{thm:dimension-main}
For the boundary in $\C^2$ and the Euclidean metric,
\[
 \mathrm{H.dim}\,\partial\KM=4.
\]
\end{theorem}

More precisely, for every open $\mathcal U\subset\C$ with
$\mathcal U\cap\partial\M\ne\varnothing$,
\[
 \mathrm{H.dim}\,\bigl(\MK_{\partial\M}
 \cap(\mathcal U\times\C)\bigr)=4.
\]
We obtain this estimate by combining Shishikura's results with a
slicing inequality. Since $\MK_{\partial\M}\subset\partial\KM$,
the theorem follows.

\subsection*{Organization}

Section~\ref{sec:basic} proves Theorems~\ref{thm:main} and
\ref{thm:stable-main}. Section~\ref{sec:puzzles} constructs the mixed
puzzles and proves Theorem~\ref{thm:puzzle}.
Section~\ref{sec:boundary} proves Theorems~\ref{thm:parabolic-main}
and~\ref{thm:dimension-main}, in
Subsections~\ref{sec:parabolic} and~\ref{sec:dimension}, respectively.

\section{Basic properties}\label{sec:basic}
\subsection{Preliminary facts}
For $a\in\C$ and $r>0$, write $D(a,r)=\{w\in\C:|w-a|<r\}$ and
$\widehat\C=\C\cup\{\infty\}$. Subsets carry their subspace
topologies; closures and boundaries are taken in the stated ambient
spaces. Local connectedness at a point means the existence of a basis
of relatively open connected neighborhoods, termed \emph{open local
connectivity} in \cite{SchleicherFibers}.

The Mandelbrot set $\M$ is compact and connected, and
\cite{DH,MilnorBook}
\[
 \begin{cases}
 K_c\text{ is connected},&c\in\M,\\
 K_c=J_c\text{ is a Cantor set},&c\notin\M.
 \end{cases}
\]
Hyperbolic Julia sets move holomorphically \cite{MSS,McMullen},
holomorphic motions over a disk extend to $\widehat{\C}$
by S\l{}odkowski's theorem \cite{Slodkowski}
(see also \cite[Theorem~3.1]{Roesch}), and
$K_c$ is locally connected for hyperbolic $c\in\M$
\cite[Part II, Expos\'e X, Section 1]{DH}.

\begin{lemma}[Uniform escape radius]\label{lem:escape}
For every $A\ge0$ there exists $R>0$ such that, for all $|c|\le A$,
\[
 K_c\subset\overline{D(0,R)},\qquad
 K_c=\bigcap_{k\ge0} f_c^{-k}\bigl(\overline{D(0,R)}\bigr).
\]
\end{lemma}
\begin{proof}
Choose $R>2$ with $R^2-A>2R$. For $|c|\le A$ and $|w|\ge R$,
\[
 |f_c(w)|\ge |w|^2-A>2|w|.
\]
Consequently, $|f_c^k(z)|\ge R$ for some $k$ implies $z\notin K_c$.
Conversely, $|f_c^k(z)|\le R$ for every $k\ge0$ implies $z\in K_c$.
\end{proof}

\begin{lemma}[Connected fibers]\label{lem:monotone}
Let $p:X\to B$ be a continuous surjection, where $X$ is compact and
$B$ is connected and Hausdorff. If $p^{-1}(\{b\})$ is connected
for every $b\in B$, then $X$ is connected.
\end{lemma}
\begin{proof}
Suppose $X=A\sqcup D$ is a separation into nonempty open sets.
Both $A$ and $D$ are compact. Connectedness of each fiber gives
\[
 p(A)\cap p(D)=\varnothing,\qquad p(A)\cup p(D)=B.
\]
Since $B$ is Hausdorff, $p(A)$ and $p(D)$ are nonempty closed sets.
They therefore form a separation of $B$, a contradiction.
\end{proof}

\subsection{Closedness, invariance, and connectedness}\label{sec:global}
\begin{proof}[Proof of Theorem~\ref{thm:main}]
The complete invariance of $K_c$ under $f_c$ gives $F^{-1}(\MK)=\MK=F(\MK)$.

Let $(c_j,z_j)\in\MK$ with $(c_j,z_j)\to(c,z)$. Choose $A>0$ such that $|c|\leq A$ and $|c_j|\leq A$ for every $j$.
Let $R$ be given by Lemma~\ref{lem:escape}.
Continuity of $(c,z)\mapsto f_c^k(z)$ yields
\[
 |f_c^k(z)|=\lim_{j\to\infty}|f_{c_j}^k(z_j)|\le R
 \qquad(k\ge0).
\]
Thus $z\in K_c$, so $\MK$ is closed. 

If $B\subset\C$ is compact,
the same lemma gives a uniform $R$ such that
\[
 (\pr|_{\MK})^{-1}(B)=\MK_B
 \subset B\times\overline{D(0,R)}.
\]
The left-hand side is closed in the compact set on the right, proving
properness. Every $f_c$ has a fixed point in $K_c$, so
$\pr(\MK)=\C$.

Since $\pr|_{\MK}$ is proper and $\M$ is compact,
$
 \KM=(\pr|_{\MK})^{-1}(\M)
$
is compact. The continuous surjection
$
 p=\pr|_{\KM}:\KM\to\M
$
has connected fibers $p^{-1}(\{c\})=\{c\}\times K_c$ for all
$c\in\M$. Lemma~\ref{lem:monotone} therefore implies that $\KM$ is connected.

To prove that $\MK$ is connected, we attach each $\bx\in\MK\setminus\KM$ to $\KM$ by a continuum. 
Write $\bx=(c_0,z_0)$, choose $a\in\M$ with
$|c_0-a|=\dist(c_0,\M)=d>0$, and set
\[
 \gamma:[0,1]\to\C,\qquad
 \gamma(t)=(1-t)c_0+ta.
\]
Thus
$
 \gamma([0,1))\subset\C\setminus\M$ and $\gamma(1)=a\in\M.
$

For each $b\in\gamma([0,1))$, hyperbolic $J$-stability
\cite{MSS} gives a disk
$\mathcal U_b\Subset\C\setminus\M$ containing $b$ and a jointly
continuous holomorphic motion
\[
 h_b:\mathcal U_b\times J_b\to\C,
 \qquad h_b(b,\zeta)=\zeta,
\]
whose slices $h_b(c,\cdot):J_b\to J_c$ are homeomorphisms.
Whenever $\gamma([u,v])\subset\mathcal U_b$ and
$w\in J_{\gamma(u)}$, define
\begin{equation}\label{eq:local-transport}
 z(t)=h_b\!\left(\gamma(t),
    h_b(\gamma(u),\cdot)^{-1}(w)\right),\qquad t \in [u,v].
\end{equation}
Then $z(u)=w$ and $z(t)\in J_{\gamma(t)}$, and $z$ is continuous.

Set $t_k=1-2^{-k}$ for $k\ge0$. Each interval
$[t_k,t_{k+1}]$ admits a finite partition
$t_k=s_0<\cdots<s_\ell=t_{k+1}$ such that
$\gamma([s_{i-1},s_i])\subset\mathcal U_{b_i}$.
Starting from $z(0)=z_0\in K_{c_0}=J_{c_0}$, apply
\eqref{eq:local-transport} successively on these subintervals, with
$w=z(s_{i-1})$. The formulas agree at their endpoints and define
a continuous lift
\[
 L:[0,1)\to\MK,\qquad
 L(t)=(\gamma(t),z(t)),\qquad L(0)=\bx.
\]
By properness,
$
 \mathfrak C:=\overline{L([0,1))}
 \subset\MK_{\gamma([0,1])}
$
is compact. It is connected because $L([0,1))$ is connected.
A subsequence of $L(t_k)$ converges to some $(a,w)\in\mathfrak C$;
since $a\in\M$, this limit lies in $\mathfrak C\cap\KM$.

Only a cluster point of $L(t)$ as $t\to1$ is needed. Therefore
\[
 \MK=\KM\cup
 \left(\bigcup_{\bx\in\MK\setminus\KM}\mathfrak C_{\bx} \right),
 \qquad \mathfrak C_{\bx}\cap\KM\ne\varnothing,
\]
where each $\mathfrak C_{\bx}$ is connected. Thus the union is connected.
\end{proof}

\subsection{Local products at stable parameters}\label{sec:stable}

The following proposition gives a local product structure over hyperbolic parameters.

\begin{proposition}[Local product structure]\label{prop:stable-product}
If $f_{c_0}$ is hyperbolic, there exist a disk $\mathcal U\ni c_0$
and a homeomorphism
\[
 T:\mathcal U\times K_{c_0}\to\MK_{\mathcal U}
\]
such that $\pr(T(c,z))=c$ and $T(c_0,z)=(c_0,z)$ for all
$(c,z)\in\mathcal U\times K_{c_0}$.
\end{proposition}
\begin{proof}
By $J$-stability \cite{MSS}, the Julia sets move holomorphically over
a sufficiently small disk $\mathcal U\ni c_0$. Adjoin the fixed
point $\infty$ and apply S\l{}odkowski's extension theorem
\cite{Slodkowski} (see also \cite[Theorem~3.1]{Roesch}).
The $\lambda$-lemma \cite[p.~193]{MSS} gives joint continuity
of the resulting holomorphic motion
\[
 h:\mathcal U\times\widehat\C\to\widehat\C,
 \qquad h^c(z):=h(c,z),
\]
normalized by $h^{c_0}=\operatorname{id}$, with each $h^c$ a sphere
homeomorphism fixing $\infty$ and satisfying $h^c(J_{c_0})=J_c$. 
It follows that $h^c(\C\setminus K_{c_0})=\C\setminus K_c$, and $ h^c(K_{c_0})=K_c$.
Thus $T:\mathcal U\times K_{c_0}\to\MK_{\mathcal U}$, defined by
$T(c,z)=(c,h^c(z))$, is a continuous bijection.
To verify continuity of $T^{-1}$, let $(c_j,w_j)\to(c,w)$ in
$\MK_{\mathcal U}$ and put $z_j=(h^{c_j})^{-1}(w_j)\in K_{c_0}$.
Every convergent subsequence $z_{j_k}\to z_*$ satisfies
\[
 h^c(z_*)=\lim_{k\to\infty}h^{c_{j_k}}(z_{j_k})=w,
\]
so $z_*=(h^c)^{-1}(w)$. Compactness of $K_{c_0}$ therefore gives
$z_j\to (h^c)^{-1}(w)$.
\end{proof}

\begin{proof}[Proof of Theorem~\ref{thm:stable-main}]
For $c_0\notin\M$, choose the disk $\mathcal{U}$ in
Proposition~\ref{prop:stable-product} inside $\C\setminus\M$.
Let $\pi_2:\mathcal U\times K_{c_0}\to K_{c_0}$ be the
second-coordinate projection. If $\mathfrak V\subset\MK_{\mathcal U}$
were a connected neighborhood of $(c_0,z)$, then
$\pi_2(T^{-1}(\mathfrak V))$ would be a connected subset of the
Cantor set $K_{c_0}$ and hence equal to $\{z\}$.
Thus $T^{-1}(\mathfrak V)\subset\mathcal U\times\{z\}$.
Since $K_{c_0}$ has no isolated points, this fiber has empty interior
in $\mathcal U\times K_{c_0}$, contradicting that
$T^{-1}(\mathfrak V)$ is a neighborhood.

For hyperbolic $c_0\in\M$, choose $\mathcal U$ inside its hyperbolic
component. Both $\mathcal U$ and $K_{c_0}$ are locally connected
\cite[Part II, Expos\'e X, Section 1]{DH}. Products of their open
connected neighborhoods form an open connected neighborhood basis
in $\mathcal U\times K_{c_0}$. In both cases,
$\MK_{\mathcal U}$ is open in $\MK$, so the homeomorphism $T$
transfers the asserted local properties to $\MK$.
\end{proof}
\section{Mixed puzzle tubes and shrinking}\label{sec:puzzles}
\subsection{Yoccoz puzzles}
We use the puzzle and parapuzzle constructions of
\cite[Sections~1--2]{Roesch}, recalling the notation needed below.
Fix a rational wake $\mathcal W_{p/q}$, where $(p,q)=1$ and $q\ge2$,
bounded by the parameter rays of angles $t^-$ and $t^+$ landing at its
root. Set
$
 \Theta=\{2^jt^+\bmod 1:0\le j<q\}.
$
For $c\in\mathcal W_{p/q}$, the dynamical rays with angles in $\Theta$
land at the repelling dividing fixed point $\alpha_c$
\cite[Theorem 1.3]{Roesch}.

Let $G_c$ be the Green function, normalized by
$G_c(z)-\log|z|\to0$ as $z\to\infty$, so that
$G_c\circ f_c=2G_c$. Let $\varphi_c$ be the normalized B\"ottcher coordinate with
\[
 \varphi_c\circ f_c=(\varphi_c)^2,\qquad
 \frac{\varphi_c(z)}{z} \to 1\quad(z\to\infty).
\]
For arbitrary $c$, $\varphi_c$ is conformal from
$\{G_c>G_c(0)\}$ onto $\{|w|>e^{G_c(0)}\}$.
If $c\in\M$, $\varphi_c$ is conformal from $\C\setminus K_c$ to $\C \setminus \overline{\mathbb D}$.
Write $R_c(\theta)$ for the ray of angle $\theta\in\R/\mathbb Z$ (argument $2\pi\theta$) in this
coordinate, continued along gradient lines wherever no critical point
is encountered. Following \cite[Definition 1.4]{Roesch}, define $ X^c=\{z:G_c(z)<1\}$, and
\[
 \begin{gathered}
 I_0^c=\partial X^c\cup
 \left(X^c\cap\bigcup_{\theta\in\Theta}\overline{R_c(\theta)}\right),
 \qquad I_m^c=f_c^{-m}(I_0^c).
 \end{gathered}
\]
The depth-$m$ puzzle pieces are the components of
$\{G_c<2^{-m}\}\setminus I_m^c$; denote their collection by $\mathscr P_m^c$.
Their boundaries consist of ray segments and arcs of
$\{G_c=2^{-m}\}$. If $c\in\M$ and
$f_c^k(0)\ne\alpha_c$ for every $k\ge0$, these pieces are Jordan
domains at every depth. For $z\in\{G_c<2^{-m}\}\setminus I_m^c$,
let $P_m^c(z)$ be the unique element of $\mathscr P_m^c$ containing $z$.
Whenever both sides are defined, we have
$
 P_{m+1}^c(z)\subset P_m^c(z).
$

\begin{figure}[htbp]
\centering
\IfFileExists{puzzle.png}{%
\includegraphics[width=.78\linewidth]{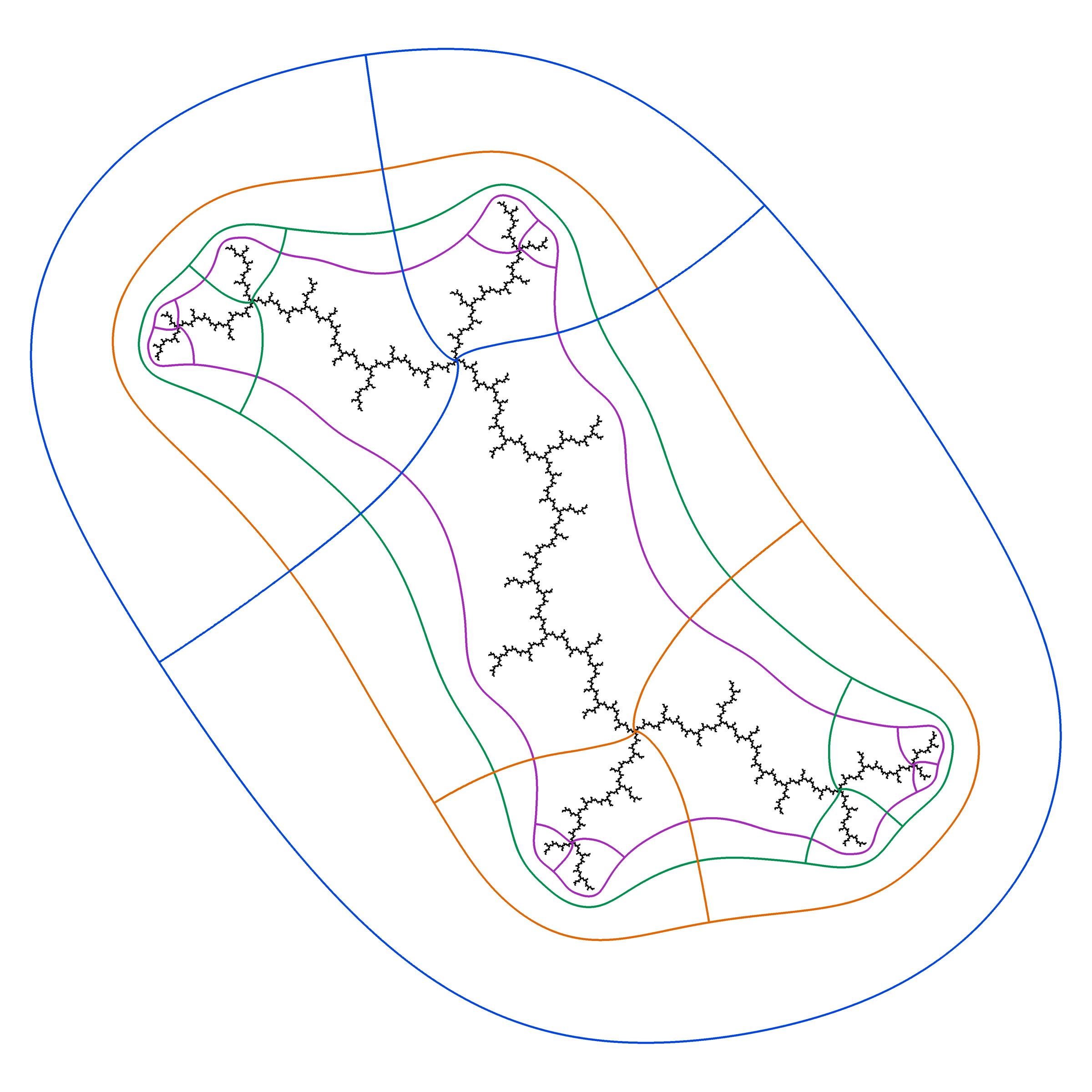}%
}{\fbox{\parbox[c][4cm][c]{.78\linewidth}{\centering
Figure file: \texttt{puzzle.png}}}}
\caption{Numerical illustration of the puzzle graphs for $f_i(z)=z^2+i$
at depths $0,1,2,3$.}
\label{fig:dynamical-puzzles}
\end{figure}

\subsection{Parapuzzles and holomorphic motions}
The parameter Green function is
$G_{\M}(c)=G_c(c)=2G_c(0)$ for $c\in\C$. Define
\[
 \Phi_{\M}:\C\setminus\M\to\C\setminus\overline{\mathbb D},
 \qquad \Phi_{\M}(c)=\varphi_c(c).
\]
The parameter external ray of angle $\theta$ is  
\[
 \mathcal R(\theta)=\Phi_{\M}^{-1}
 \{re^{2\pi i\theta}:r>1\}.
\]
For $n\ge0$, set $ \Theta_n=\{\theta\in\R/\mathbb Z:2^n\theta\in\Theta\} $, and
\[
 \mathcal X_n=\{c\in\mathcal W_{p/q}:G_{\M}(c)<2^{-n}\},\qquad
 \mathcal I_n=\partial\mathcal X_n\cup
 \left(\mathcal X_n\cap
 \bigcup_{\theta\in\Theta_n}\overline{\mathcal R(\theta)}\right).
\]
The components of $\mathcal X_n\setminus\mathcal I_n$ are bounded,
simply connected parapuzzle pieces
\cite[Definition 2.1 and Remark 2.4(b)]{Roesch}.

Fix $c_0\in\M\cap\mathcal W_{p/q}$ such that
$f_{c_0}^k(0)\ne\alpha_{c_0}$ for every $k\ge0$. Roesch's Lemma 2.3
gives
\[
 c\in\mathcal I_n\quad\Longleftrightarrow\quad c\in I_n^c
 \qquad(c\in\mathcal W_{p/q}).
\]
Since $c_0\in K_{c_0}$, the critical-orbit assumption implies
$c_0\notin\mathcal I_n$ for every $n$. Let
$\mathcal P_n=\mathcal P_n(c_0)$ be the parapuzzle piece containing
$c_0$. For $c\in\mathcal P_n$,
\[
 c\notin\bigcup_{j=0}^n I_j^c,
 \qquad 0\notin I_{n+1}^c,
 \qquad G_c(0)=\tfrac12G_{\M}(c)<2^{-(n+1)}.
\]

By \cite[Lemma~2.5 and its construction, pp.~123--124]{Roesch}
and S\l{}odkowski's extension theorem \cite{Slodkowski}
(see also \cite[Theorem~3.1]{Roesch}), we may choose the following
finite-depth motions. The extension theorem applies after identifying
the simply connected domain $\mathcal P_n$ conformally with a disk.
For every $n\ge0$ there is a holomorphic motion
\[
 h_n:\mathcal P_n\times\widehat\C\to\widehat\C,
 \qquad h_n^c(z):=h_n(c,z),
\]
normalized by $h_n^{c_0}=\operatorname{id}$, such that each
$h_n^c$ is a sphere homeomorphism fixing $\infty$ and
\[
 \begin{gathered}
 h_n^c(I_{n+1}^{c_0})=I_{n+1}^c,\qquad
 h_n^c(z)=\varphi_c^{-1}(\varphi_{c_0}(z))
 \quad\text{if }G_{c_0}(z)\ge2^{-(n+1)}.
 \end{gathered}
\]
The motions agree on common marked graph and exterior points.
The extensions on complementary domains need not be unique.
All motions used below have open parameter domains $\mathcal P_n$.

\begin{figure}[htbp]
\centering
\IfFileExists{para-puzzle.png}{%
\includegraphics[width=.98\linewidth]{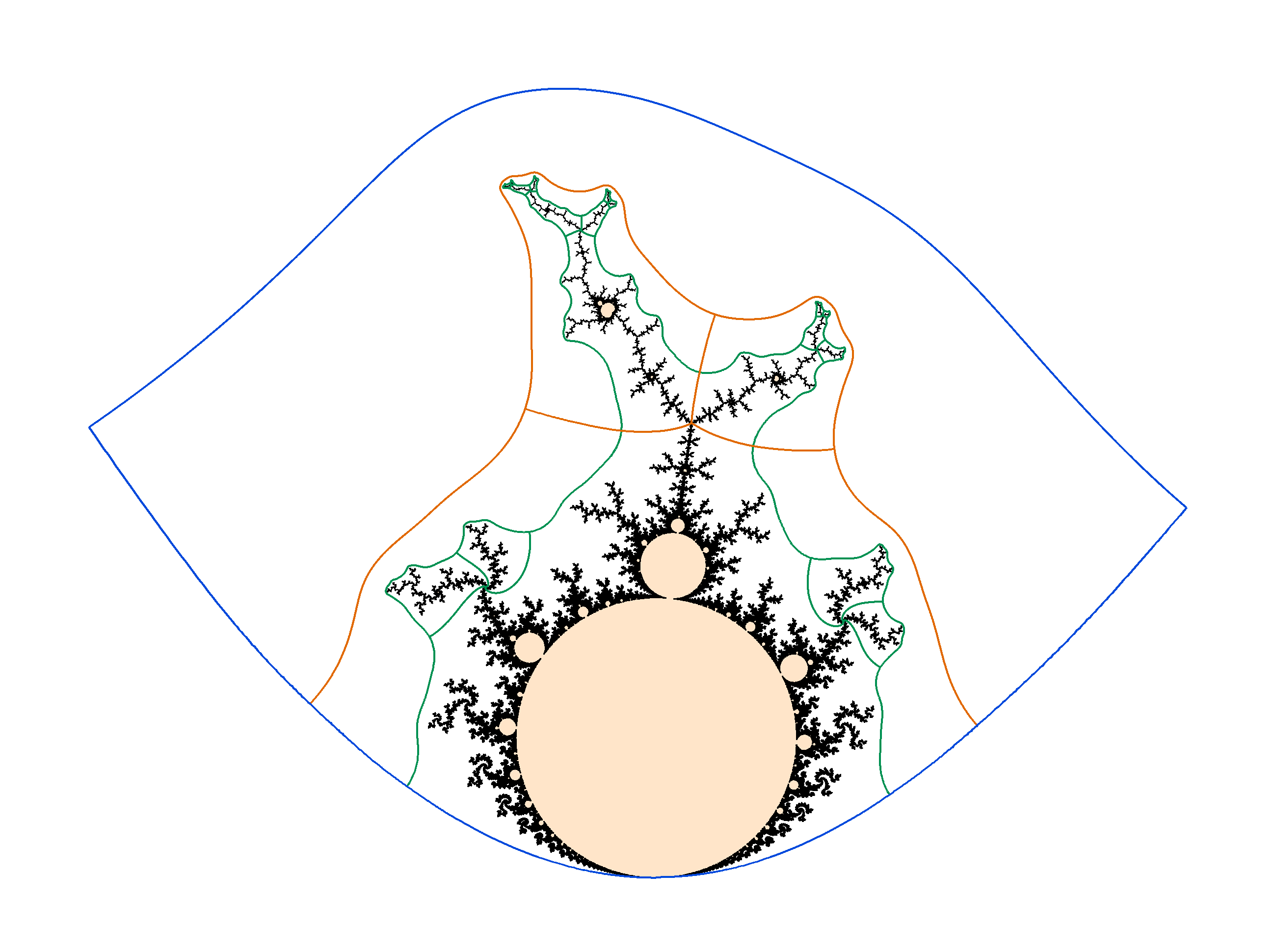}%
}{\fbox{\parbox[c][4cm][c]{.98\linewidth}{\centering
Figure file: \texttt{para-puzzle.png}}}}
\caption{Numerical illustration of the parapuzzle graphs in the $1/3$-wake
at depths $0,3,6$.}
\label{fig:parameter-puzzles}
\end{figure}

\subsection{Mixed tubes}
Fix $z_0\in K_{c_0}\setminus\bigcup_{j\ge0}I_j^{c_0}$ and set
$P_m:=P_m^{c_0}(z_0)$ for $m\ge0$. These pieces are uniquely
determined and satisfy $z_0\in P_{m+1}\subset P_m$.
For integers $n\ge m\ge0$, define
\begin{equation}\label{eq:two-index-tube}
 \begin{gathered}
 \mathfrak P_{n,m}(c_0,z_0)
 =\{(c,w):c\in\mathcal P_n,\ w\in h_n^c(P_m)\},
 \qquad
 \mathfrak P_n:=\mathfrak P_{n,n}.
 \end{gathered}
\end{equation}
Unless specified otherwise, closures of mixed tubes are taken in $\C^2$.

The planar nests satisfy $\mathcal P_{n+1}\subset\mathcal P_n$
and $P_{m+1}\subset P_m$ by the standard puzzle and parapuzzle
constructions \cite[Definitions~1.4 and~2.1]{Roesch}.

The compatibility of the marked motions
\cite[Lemma~2.5 and its construction]{Roesch} gives the following
identities for the corresponding puzzle pieces.
If $n'\ge n\ge m$ and $c\in\mathcal P_{n'}$, then
\begin{equation}\label{eq:motion-compatibility}
 h_{n'}^c(P_m)=h_n^c(P_m),\qquad
 h_{n'}^c(\overline{P_m})=h_n^c(\overline{P_m}).
\end{equation}
By \eqref{eq:motion-compatibility}, $h_m^c(P_m)=h_n^c(P_m)$ for $n\ge m$ and
$c\in\mathcal P_n$; hence it is independent of the sphere extension.
The impressions are
\[
 \begin{gathered}
 \mathcal I(c_0)=\bigcap_{n\ge0}\overline{\mathcal P_n},\qquad
 I_{c_0}(z_0)=\bigcap_{m\ge0}\overline{P_m},\qquad
 \mathfrak I(c_0,z_0)=\bigcap_{n\ge0}\overline{\mathfrak P_n}.
 \end{gathered}
\]

\begin{lemma}[Properties of mixed tubes]\label{lem:puzzle-structure}
We have the following properties:
\begin{enumerate}
\item[\textup{(i)}]  For $ 0\le m\le m'\le n $, we have
\[
 \mathfrak P_{n+1}\subset\mathfrak P_n,\qquad
 \mathfrak P_{n,m'}\subset\mathfrak P_{n,m}.
\]
\item[\textup{(ii)}] For every pair of integers $n\ge m\ge0$, the set
$\overline{\mathfrak P_{n,m}}$ is compact.
\item[\textup{(iii)}] For every $n\ge m\ge0$,
\[
 \pr(\overline{\mathfrak P_{n,m}})=\overline{\mathcal P_n},\qquad
 \overline{\mathfrak P_{n,m}}\cap(\{c_0\}\times\C)
 =\{c_0\}\times\overline{P_m}.
\]
\end{enumerate}
\end{lemma}
\begin{proof}
\textup{(i)} The inclusions follow from planar nesting and
\eqref{eq:motion-compatibility}.

\textup{(ii)} The parameter pieces lie in the bounded set
$\mathcal P_0$, and the maximum principle gives $G_c(w)\le2^{-m}\le1$
on each tube. For bounded $c$, the escape estimate
$|f_c(w)|\ge |w|^2/2$ for large $|w|$ gives, by iteration,
$G_c(w)\ge\log(|w|/2)$. Thus $G_c(w)\to\infty$ uniformly as
$|w|\to\infty$, so the tubes are bounded in $\C^2$ and their
closures are compact.

\textup{(iii)}  Since
$\pr(\mathfrak P_{n,m})=\mathcal P_n$, we have $ \pr(\overline{\mathfrak P_{n,m}}) \subset\overline{\mathcal P_n} $.
By (ii), $\overline{\mathfrak P_{n,m}}$ is compact.
It follows that $\overline{\mathcal P_n} \subset\pr(\overline{\mathfrak P_{n,m}})$.
For the slice identity, let
\[
 (c_j,h_n^{c_j}(w_j))\to(c_0,w),
 \qquad c_j\in\mathcal P_n,\quad w_j\in P_m.
\]
A subsequence satisfies $w_j\to w_*\in\overline{P_m}$. Joint
continuity and $h_n^{c_0}=\operatorname{id}$ imply
$w=h_n^{c_0}(w_*)=w_*$. Conversely, approximating each
$w\in\overline{P_m}$ by points of $P_m$ at the fixed parameter
$c_0$ gives the reverse inclusion.
\end{proof}

\begin{proof}[Proof of Theorem~\ref{thm:puzzle}\textup{(i)}]
Choose a rational wake for the non-degenerate pair $(c_0,z_0)$. The critical-orbit assumption gives
$c_0\notin\mathcal I_n$ for
all $n$, so the parameter nest is defined. Graph avoidance gives
$z_0\in P_n$ for all $n$.
Suppose $\mathcal I(c_0)=\{c_0\}$ and $I_{c_0}(z_0)=\{z_0\}$.
The projection identity in Lemma~\ref{lem:puzzle-structure}\textup{(iii)}
gives $\pr(\mathfrak I(c_0,z_0))\subset\mathcal I(c_0)=\{c_0\}$.
Intersecting its central slice identities therefore yields
\[
 \mathfrak I(c_0,z_0)
 =\mathfrak I(c_0,z_0)\cap(\{c_0\}\times\C)
 =\{c_0\}\times I_{c_0}(z_0)=\{(c_0,z_0)\}.
\]

Conversely, assume $\mathfrak I(c_0,z_0)=\{(c_0,z_0)\}$.
For $c\in\mathcal I(c_0)$, the sets
$L_n(c):=\overline{\mathfrak P_n}\cap(\{c\}\times\C)$
are nonempty, compact, and nested by
Lemma~\ref{lem:puzzle-structure}. Thus
\[
 \varnothing\ne\bigcap_nL_n(c)
 =\mathfrak I(c_0,z_0)\cap(\{c\}\times\C),
\]
which forces $c=c_0$. Hence $\mathcal I(c_0)=\{c_0\}$.
Intersecting the central slice identities gives
\[
 \{c_0\}\times I_{c_0}(z_0)
 =\mathfrak I(c_0,z_0)\cap(\{c_0\}\times\C)
 =\{(c_0,z_0)\},
\]
so $I_{c_0}(z_0)=\{z_0\}$.
\end{proof}

\subsection{Local connectedness from mixed shrinking}\label{sec:closed}
Retain the parapuzzle pieces $\mathcal P_m$, dynamical pieces $P_m$,
and motions $h_m$ defined above.

\begin{proposition}[Open mixed neighborhoods]\label{prop:open-mixed}
Suppose $\mathcal I(c_0)=\{c_0\}$ and
$z_0\in J_{c_0}\setminus\bigcup_{j\ge0}I_j^{c_0}$.
For each $m\ge0$ and all sufficiently large $n\ge m$,
$\mathfrak P_{n,m}\cap\KM$ is a connected relatively open
neighborhood of $(c_0,z_0)$ in $\KM$.
\end{proposition}
\begin{proof}
Choose a repelling periodic point $q_0\in J_{c_0}\cap P_m$.
Since $\overline{\mathcal P_n}$ shrinks to $c_0$, the implicit
function theorem and joint continuity of the motion give, for all
sufficiently large $n$, a holomorphic continuation $q$ on
$\mathcal P_n$ with $q(c)\in J_c\cap h_m^c(P_m)$.
By \eqref{eq:motion-compatibility},
\[
 \mathfrak P_{n,m}\cap\KM
 =\{(c,w):c\in\mathcal P_n\cap\M,\ w\in K_c\cap h_m^c(P_m)\}.
\]
The base $\mathcal P_n\cap\M$ is connected: parapuzzle pieces
cut $\M$ along finitely many rational ray pairs; see the connectedness
argument in \cite[proof of Proposition~3.5]{SchleicherFibers}.
Each fiber $K_c\cap h_m^c(P_m)$ is connected by planar puzzle
connectivity and meets the connected graph
$\{(c,q(c)):c\in\mathcal P_n\cap\M\}$. Their union is therefore
connected. Joint continuity of $(c,w)\mapsto(h_m^c)^{-1}(w)$
(which follows from joint continuity of $h_m$ and compactness of
$\widehat\C$) gives relative openness. Finally, the set contains
$(c_0,z_0)$.
\end{proof}

\begin{proof}[Proof of Theorem~\ref{thm:puzzle}\textup{(ii)}]
For the non-degenerate pair $(c_0,z_0)$, assume
$\mathcal I(c_0)=\{c_0\}$ and $I_{c_0}(z_0)=\{z_0\}$.
First, $z_0\in J_{c_0}$. Otherwise $z_0$ lies in a bounded Fatou
component $U_0$. Since each puzzle graph avoids $U_0$,
$U_0\subset P_m\ (m\ge0),\  U_0\subset I_{c_0}(z_0)=\{z_0\}$,
a contradiction.

Let $U$ be a neighborhood of $(c_0,z_0)$ in $\KM$. Choose
$\varepsilon>0$ such that
$\KM\cap(D(c_0,\varepsilon)\times D(z_0,\varepsilon))\subset U$.
Dynamical shrinking gives $m$ with
$\overline{P_m}\subset D(z_0,\varepsilon)$. By joint continuity
of $h_m$, compactness of $\overline{P_m}$, and parameter shrinking,
for all sufficiently large $n\ge m$ we have
$\mathcal P_n\subset D(c_0,\varepsilon)$ and
$h_m^c(\overline{P_m})\subset D(z_0,\varepsilon)$ for every
$c\in\mathcal P_n$. Proposition~\ref{prop:open-mixed} then gives
a connected relatively open neighborhood
$\mathfrak P_{n,m}\cap\KM\subset U$ of $(c_0,z_0)$.
This proves local connectedness of $\KM$ at $(c_0,z_0)$.
\end{proof}

\section{Non-local connectivity and Hausdorff dimension}
\label{sec:boundary}

\subsection{Non-local connectivity}\label{sec:parabolic}

For $b\in\M$ and $x,y\in\M\setminus\{b\}$, say that $b$
\emph{separates} $x$ and $y$ if they belong to distinct connected
components of $\M\setminus\{b\}$. Equivalently, every connected
set $C\subset\M$ with $x,y\in C$ contains $b$.

\begin{lemma}[Gate criterion]\label{lem:gate}
Let $(c,z_*)\in\KM$. Suppose that $a_n,b_n\in\M$ and
$z_n\in K_{a_n}$ satisfy
\[
 (a_n,z_n)\to(c,z_*),\qquad b_n\to c,
\]
and that $b_n$ separates $c$ and $a_n$ for every $n\ge1$.
If there exist $\rho>0$ and $N\ge1$ such that
\[
 K_{b_n}\cap D(z_*,\rho)=\varnothing\qquad(n\ge N),
\]
then $\KM$ is not locally connected at $(c,z_*)$.
\end{lemma}

\begin{proof}
Suppose that $\KM$ is locally connected at $p_*=(c,z_*)$.
There is a relatively open connected set $\mathfrak N\subset\KM$ with
$p_*\in\mathfrak N\subset\KM\cap\bigl(\C\times D(z_*,\rho)\bigr)$.
For all sufficiently large $n$, $(a_n,z_n)\in\mathfrak N$.
The set $\pr(\mathfrak N)$ is connected and contains $c,a_n$;
hence it contains $b_n$. Consequently,
\[
 \varnothing\ne
 \{w:(b_n,w)\in\mathfrak N\}
 \subset K_{b_n}\cap D(z_*,\rho),
\]
contrary to the hypothesis.
\end{proof}

\begin{lemma}[Real-root gate]\label{lem:real-gate}
Let $b\in(-2,1/4)$ be a real parabolic root whose two characteristic
parameter rays are nonreal. Then
\[
 x,y\in\M\cap\R,\quad x<b<y
 \quad\Longrightarrow\quad
 b\text{ separates }x\text{ and }y\text{ in }\M.
\]
In particular, this holds for every real primitive root of period
greater than one.
\end{lemma}

\begin{proof}
The characteristic rays $\mathcal R(t^-)$ and
$\mathcal R(t^+)$ landing at $b$ are complex conjugate
\cite[Theorem~3.1 and Corollary~5.4]{MilnorPortrait}. Since
$t^\pm\notin\{0,1/2\}$ and the parameter B\"ottcher coordinate
commutes with conjugation, neither open ray intersects $\R$.
Thus
$\Gamma=\mathcal R(t^-) \cup\{b,\infty\}\cup\mathcal R(t^+)$
is a Jordan curve in $\widehat\C$ satisfying
\[
 \Gamma\cap(\R\cup\{\infty\})=\{b,\infty\},
 \qquad \Gamma\cap\M=\{b\}.
\]
The intervals $(-\infty,b)$ and $(b,\infty)$ lie in opposite
components of $\widehat\C\setminus\Gamma$. A connected subset of
$\M\setminus\{b\}$ lies in one of these components, so it cannot
contain both $x$ and $y$.
\end{proof}

Fix a real primitive root $c_*\ne1/4$. For
$\theta\in\R/\mathbb Z$, let $L_\theta$ denote the compact
Lavaurs filled set of phase $\theta$ defined in
\cite[Theorem~3.1]{RojasYampolsky}. These sets satisfy
\begin{equation}\label{eq:lavaurs-inclusions}
 J_{c_*}\subsetneq L_\theta\subsetneq K_{c_*}.
\end{equation}

For nonempty compact sets $E_k\subset\C$, we use the Kuratowski limits
\[
 \begin{aligned}
 \LimInf_{k\to\infty}E_k
 &=\{z\in\C:\dist(z,E_k)\to0\},\\
 \LimSup_{k\to\infty}E_k
 &=\{z\in\C:\liminf_{k\to\infty}\dist(z,E_k)=0\}.
 \end{aligned}
\]
By \cite[Theorems~3.1 and 3.3]{RojasYampolsky}, there exist phases
$\theta_1,\theta_2\in[0,1)$ and strictly decreasing sequences of
real primitive roots $c_k^{(i)}\searrow c_*$, $i=1,2$, such that,
with $L_{c_*}^{(i)}=L_{\theta_i}$,
\begin{equation}\label{eq:lavaurs-separation}
 \partial L_{c_*}^{(1)}\setminus L_{c_*}^{(2)}\ne\varnothing
\end{equation}
and
\begin{equation}\label{eq:lavaurs-limits}
 \partial L_{c_*}^{(1)}\subset
 \LimInf_{k\to\infty}J_{c_k^{(1)}},
 \qquad
 \LimSup_{k\to\infty}K_{c_k^{(2)}}\subset L_{c_*}^{(2)}.
\end{equation}
Fix
\begin{equation}\label{eq:zstar}
 z_*\in\partial L_{c_*}^{(1)}\setminus L_{c_*}^{(2)}
 \subset K_{c_*}\setminus J_{c_*},
\end{equation}
where the inclusion follows from \eqref{eq:lavaurs-inclusions}.

\begin{proof}[Proof of Theorem~\ref{thm:parabolic-main}]
Set $c_*=-7/4$, the primitive period-three root
\cite[Section~3]{RojasYampolsky}, and choose the preceding sequences.
The lower-limit inclusion gives
$\dist(z_*,J_{c_k^{(1)}})\to0$. Choose $z_k\in J_{c_k^{(1)}}$
realizing this distance.
Since $L_{c_*}^{(2)}$ is compact, set
$d=\dist(z_*,L_{c_*}^{(2)})>0,\  \rho=d/2$.
If $K_{c_k^{(2)}}\cap D(z_*,\rho)\ne\varnothing$ for infinitely many $k$, there would be $k_j\to\infty$ and
$w_j\in K_{c_{k_j}^{(2)}}\cap D(z_*,\rho)$. After extraction,
$w_j\to w\in\overline{D(z_*,\rho)}$, whereas
\eqref{eq:lavaurs-limits} gives $w\in L_{c_*}^{(2)}$. Hence
$d\le|w-z_*|\le d/2$, a contradiction.

Recursively choose increasing indices $k_n,\ell_n$ such that
\[
 a_n=c_{k_n}^{(1)},\qquad b_n=c_{\ell_n}^{(2)},\qquad
 c_*<a_{n+1}<b_n<a_n.
\]
Then $a_n,b_n\to c_*$. Relabeling $z_{k_n}$ as $z_n$, we have
$z_n\in J_{a_n}$ and
\[
 (a_n,z_n)\to(c_*,z_*),\qquad
 K_{b_n}\cap D(z_*,\rho)=\varnothing
\]
for all sufficiently large $n$.
By Lemma~\ref{lem:real-gate}, $b_n$ separates $c_*$ and $a_n$.
Lemma~\ref{lem:gate} therefore gives non-local connectivity of
$\KM$ at $(c_*,z_*)$.

The Mandelbrot set is locally connected at primitive parabolic
parameters \cite{TanParabolic,SchleicherFibers}. Since $f_{c_*}$ is
geometrically finite, $J_{c_*}$ and $K_{c_*}$ are locally connected
\cite{DH,MilnorBook}. Finally, \eqref{eq:zstar} gives
$z_*\in K_{c_*}\setminus J_{c_*}$.
\end{proof}

\subsection{Hausdorff dimension of the boundary}\label{sec:dimension}

All Hausdorff dimensions in this section use the Euclidean metric.

For $c\in\C$, define
\[
 \operatorname{hypdim}(f_c)
 =\sup\{\mathrm{H.dim}\,X:X\text{ is a hyperbolic set for }f_c\}.
\]
Here $X\subset\C$ is hyperbolic if it is compact,
$f_c(X)\subset X$, and there exist $A>0$ and $\lambda>1$ such that
\[
 |(f_c^n)'(z)|\ge A\lambda^n
 \qquad(z\in X,\ n\ge0).
\]
We use the following results of Shishikura \cite{Shishikura}:
\begin{enumerate}
\item For every open $\mathcal V\subset\C$ with
$\mathcal V\cap\partial\M\ne\varnothing$,
\[
 \mathrm{H.dim}\,(\mathcal V\cap\partial\M)=2
 \qquad\text{(Theorem A).}
\]
\item For $s<2$, the set
$\mathcal V_s=\{c\in\C:\operatorname{hypdim}(f_c)>s\}$
is open, and $\mathcal V_s\cap\partial\M$ is dense in
$\partial\M$ (Property~(1.3) and the proof of Theorem B).
\item For every $c\in\C$,
\[
 \operatorname{hypdim}(f_c)\le\mathrm{H.dim}\,J_c
 \qquad\text{(Property~(1.1)).}
\]
\end{enumerate}

\begin{proof}[Proof of Theorem~\ref{thm:dimension-main}]
First,
$\MK_{\partial\M}\subset\partial\KM$.
Indeed, if $(c,z)\in\MK_{\partial\M}$, choose
$c_j\in\C\setminus\M$ with $c_j\to c$. Then
$(c_j,z)\notin\KM,\ (c_j,z)\to(c,z)\in\KM$,
and $\KM$ is closed.

Fix an open $\mathcal U\subset\C$ meeting $\partial\M$ and
$0<s<2$. Choose
\[
 c_s\in\mathcal U\cap\mathcal V_s\cap\partial\M,
 \qquad
 \overline{D(c_s,r)}\subset\mathcal U\cap\mathcal V_s
\]
for some $r>0$. Set
$\mathcal A_s=\partial\M\cap\overline{D(c_s,r)},\  \mathfrak E_s=\MK_{\mathcal A_s}$.
The set $\mathcal A_s$ is compact, and Theorem A gives
\[
 2=\mathrm{H.dim}\,(\partial\M\cap D(c_s,r))
 \le\mathrm{H.dim}\,\mathcal A_s\le2.
\]
For every $c\in\mathcal A_s$,
\[
 \mathrm{H.dim}\,K_c\ge\mathrm{H.dim}\,J_c
 \ge\operatorname{hypdim}(f_c)>s.
\]
By the standard slicing inequality \cite[Theorem~2.1]{HeraKeletiMathe},
\[
 \mathrm{H.dim}\,\bigl(\MK_{\partial\M}
 \cap(\mathcal U\times\C)\bigr)
 \ge\mathrm{H.dim}\,\mathfrak E_s\ge\mathrm{H.dim}\,\mathcal A_s+s=2+s.
\]
Letting $s\nearrow2$ gives
\[
 \mathrm{H.dim}\,\bigl(\MK_{\partial\M}
 \cap(\mathcal U\times\C)\bigr)=4.
\]
Taking $\mathcal U=\C$, we conclude that
\[
 4=\mathrm{H.dim}\,\MK_{\partial\M}
 \le\mathrm{H.dim}\,\partial\KM\le4.
\]
\end{proof}

%The dimension statement concerns parameter cylinders $\mathcal U\times\C$; it does not assert local dimension four at every boundary point.

\subsection*{Declaration of AI use}
During the preparation of this manuscript, the author used OpenAI's
ChatGPT and Codex to assist with language editing, LaTeX preparation,
notation checks, literature searches, numerical illustrations, and
discussions of mathematical arguments. The author is responsible for
the final text, the mathematical statements and proofs, and the accuracy
of the references.


\begin{thebibliography}{AKLS09}

\bibitem[AKLS09]{AKLS}
A. Avila, J. Kahn, M. Lyubich, and W. Shen,
\emph{Combinatorial rigidity for unicritical polynomials},
Ann. of Math. (2) \textbf{170} (2009), no.~2, 783--797.

\bibitem[BC12]{BuffCheritat}
X. Buff and A. Ch\'eritat,
\emph{Quadratic Julia sets with positive area},
Ann. of Math. (2) \textbf{176} (2012), no.~2, 673--746.

\bibitem[Dou94]{DouadyContinuity}
A. Douady,
\emph{L'ensemble de Julia d\'epend-il contin\^ument du polyn\^ome?},
in \emph{Aspects des syst\`emes dynamiques}, Journ\'ees math\'ematiques
X-UPS, \'Ecole Polytechnique, Palaiseau, 1994, pp.~41--91.

\bibitem[DH84/85]{DH}
A. Douady and J. H. Hubbard,
\emph{\'Etude dynamique des polyn\^omes complexes}, Parts I and II,
Publications Math\'ematiques d'Orsay 84-02 and 85-04,
Universit\'e de Paris-Sud, Orsay, 1984--1985.

\bibitem[DH85]{DHPolynomialLike}
A. Douady and J. H. Hubbard,
\emph{On the dynamics of polynomial-like mappings},
Ann. Sci. \'Ecole Norm. Sup. (4) \textbf{18} (1985), no.~2, 287--343.

\bibitem[DSZ97]{DSZ}
A. Douady, P. Sentenac, and M. Zinsmeister,
\emph{Implosion parabolique et dimension de Hausdorff},
C. R. Acad. Sci. Paris S\'er. I Math. \textbf{325} (1997), no.~7,
765--772.

\bibitem[HKM24]{HeraKeletiMathe}
K. H\'era, T. Keleti, and A. M\'ath\'e,
\emph{A Fubini-type theorem for Hausdorff dimension},
J. Anal. Math. \textbf{152} (2024), no.~2, 471--506,
\href{https://doi.org/10.1007/s11854-023-0302-3}{doi:10.1007/s11854-023-0302-3}.

\bibitem[Hub93]{HubbardYoccoz}
J. H. Hubbard,
\emph{Local connectivity of Julia sets and bifurcation loci: three
theorems of J.-C. Yoccoz},
in \emph{Topological Methods in Modern Mathematics} (Stony Brook, NY,
1991), Publish or Perish, Houston, TX, 1993, pp.~467--511.

\bibitem[Ino09]{InouDiscontinuity}
H. Inou,
\emph{Combinatorics and topology of straightening maps II: Discontinuity},
preprint, \href{https://arxiv.org/abs/0903.4289}{arXiv:0903.4289}.

\bibitem[IK12]{InouKiwi}
H. Inou and J. Kiwi,
\emph{Combinatorics and topology of straightening maps I:
compactness and bijectivity},
Adv. Math. \textbf{231} (2012), no.~5, 2666--2733,
\href{https://doi.org/10.1016/j.aim.2012.07.014}{doi:10.1016/j.aim.2012.07.014}.

\bibitem[Jon99]{Jonsson}
M. Jonsson,
\emph{Dynamics of polynomial skew products on $\C^2$},
Math. Ann. \textbf{314} (1999), no.~3, 403--447.

\bibitem[KL09]{KahnLyubichLC}
J. Kahn and M. Lyubich,
\emph{Local connectivity of Julia sets for unicritical polynomials},
Ann. of Math. (2) \textbf{170} (2009), no.~1, 413--426.

\bibitem[Lav89]{Lavaurs}
P. Lavaurs,
\emph{Syst\`emes dynamiques holomorphes: explosion de points paraboliques},
Ph.D. thesis, Universit\'e Paris-Sud, Orsay, 1989.

\bibitem[Lyu97]{Lyubich12}
M. Lyubich,
\emph{Dynamics of quadratic polynomials, I--II},
Acta Math. \textbf{178} (1997), no.~2, 185--297.

\bibitem[Lyu00]{Lyubich3}
M. Lyubich,
\emph{Dynamics of quadratic polynomials, III: parapuzzle and SBR measures},
Ast\'erisque \textbf{261} (2000), 173--200.

\bibitem[MSS83]{MSS}
R. Ma\~n\'e, P. Sad, and D. Sullivan,
\emph{On the dynamics of rational maps},
Ann. Sci. \'Ecole Norm. Sup. (4) \textbf{16} (1983), 193--217.

\bibitem[McM94]{McMullen}
C. T. McMullen,
\emph{Complex Dynamics and Renormalization},
Annals of Mathematics Studies, vol.~135,
Princeton University Press, Princeton, NJ, 1994.

\bibitem[Mil00a]{MilnorLC}
J. Milnor,
\emph{Local connectivity of Julia sets: expository lectures},
in \emph{The Mandelbrot Set, Theme and Variations} (Tan Lei, ed.),
London Math. Soc. Lecture Note Ser., vol.~274,
Cambridge University Press, Cambridge, 2000, pp.~67--116.

\bibitem[Mil00b]{MilnorPortrait}
J. Milnor,
\emph{Periodic orbits, external rays and the Mandelbrot set:
an expository account},
Ast\'erisque \textbf{261} (2000), 277--333.

\bibitem[Mil06]{MilnorBook}
J. Milnor,
\emph{Dynamics in One Complex Variable}, 3rd ed.,
Annals of Mathematics Studies, vol.~160,
Princeton University Press, Princeton, NJ, 2006.

\bibitem[Mil14]{MilnorBoundaries}
J. Milnor,
\emph{Hyperbolic component boundaries},
lecture slides, Gyeongju, August 23, 2014, revised version,
\href{https://www.math.stonybrook.edu/~jack/HCBkoreaPrint.pdf}{available online}.

\bibitem[Roe00]{Roesch}
P. Roesch,
\emph{Holomorphic motions and puzzles (following M. Shishikura)},
in \emph{The Mandelbrot Set, Theme and Variations} (Tan Lei, ed.),
London Math. Soc. Lecture Note Ser., vol.~274,
Cambridge University Press, Cambridge, 2000, pp.~117--132.

\bibitem[RY21]{RojasYampolsky}
C. Rojas and M. Yampolsky,
\emph{Real quadratic Julia sets can have arbitrarily high complexity},
Found. Comput. Math. \textbf{21} (2021), no.~1, 59--69.

\bibitem[Sch04]{SchleicherFibers}
D. Schleicher,
\emph{On fibers and local connectivity of Mandelbrot and Multibrot sets},
in \emph{Fractal Geometry and Applications: A Jubilee of Beno\^it Mandelbrot},
Part~1, Proc. Sympos. Pure Math., vol.~72, Amer. Math. Soc., Providence,
RI, 2004, pp.~477--517.

\bibitem[Ses99]{Sester}
O. Sester,
\emph{Hyperbolicit\'e des polyn\^omes fibr\'es},
Bull. Soc. Math. France \textbf{127} (1999), no.~3, 393--428.

\bibitem[Shi98]{Shishikura}
M. Shishikura,
\emph{The Hausdorff dimension of the boundary of the Mandelbrot set
and Julia sets},
Ann. of Math. (2) \textbf{147} (1998), no.~2, 225--267.

\bibitem[Slo91]{Slodkowski}
Z. S\l{}odkowski,
\emph{Holomorphic motions and polynomial hulls},
Proc. Amer. Math. Soc. \textbf{111} (1991), no.~2, 347--355.

\bibitem[Tan00]{TanParabolic}
L. Tan,
\emph{Local properties of the Mandelbrot set at parabolic points},
in \emph{The Mandelbrot Set, Theme and Variations} (Tan Lei, ed.),
London Math. Soc. Lecture Note Ser., vol.~274,
Cambridge University Press, Cambridge, 2000, pp.~133--160.

\end{thebibliography}
\end{document}